\documentclass[a4paper,11pt]{amsart}

\usepackage[colorlinks,linkcolor=blue,citecolor=blue,bookmarksnumbered,plainpages,driverfallback=dvipdfm,backref=page]{hyperref}
\usepackage{bookmark}
\usepackage{latexsym, amssymb, amsmath, amsthm, mathrsfs, bbm, pstricks, ifthen, amsfonts}
\usepackage{orcidlink}
\usepackage{tikz}
\usepackage{tikz-cd}
\usepackage{hyperref}
\usepackage{enumerate}
\usepackage{version} 

\usepackage[all, knot]{xy}
\xyoption{arc}
\allowdisplaybreaks[4]

\def \Hom{\mathrm{Hom}}

\def \cc{\mathcal{C}}

\def \dd{\mathcal{D}}
\def \ev{\mathrm{ev}}
\def \coev{\mathrm{coev}}

\def \1{\mathbf{1}}
\def \id{\operatorname{id}}

\def \Set{\sf Set}
\def \id{\mathrm{Id}}

\def \im{\mathrm{Im}}
\def \ker{\mathrm{Ker}}
\def \coker{\mathrm{Coker}}

\newcommand{\f}{\mathcal{F}}
\newcommand{\p}{\mathcal{P}}
\newcommand{\m}{\mathcal{M}}

\numberwithin{equation}{section}

\newtheorem{theorem}{Theorem}[section]
\newtheorem{lemma}[theorem]{Lemma}
\newtheorem{proposition}[theorem]{Proposition}
\newtheorem{corollary}[theorem]{Corollary}

\newtheorem{definition}[theorem]{Definition}
\newtheorem{example}[theorem]{Example}
\newtheorem{remark}[theorem]{Remark}

\newenvironment{invisible}{{\noindent\sc \colorbox{yellow}{Invisible:}\;}\color{gray}}{\medskip}

\excludeversion{invisible}

\begin{document}

\title{Equivalent definitions of fusion category arising from separability}
\thanks{}

\author{Zhenbang Zuo~\orcidlink{0009-0008-9013-0365}}

\address{%
\parbox[b]{\linewidth}{University of Turin, Department of Mathematics ``G. Peano'', via
Carlo Alberto 10, I-10123 Torino, Italy}}
\email{zhenbang.zuo@edu.unito.it }

\subjclass[2020]{Primary 18M20; Secondary 18A22}

\begin{abstract}
For a semisimple multiring category with left duals, we prove that the unit object is simple if and only if the tensor functors by any non-zero algebra are separable (resp. faithful, resp. Maschke, resp. dual Maschke, resp. conservative). This induces a list of equivalent definitions of fusion category. As an application, we describe the connectness of a class of weak Hopf algebras by the separability of tensor functors. We also consider applications to transfer of simplicity between the unit objects, semisimple indecomposable module category and Grothendieck ring.
\end{abstract}

\keywords{Fusion category, Separable functor}
\date{}
\maketitle

\section{Introduction}

The notion of separable functor was introduced in \cite{NVV89} as a generalization of separability of rings. Separability of some particular functor often provides information of objects or morphisms in the category. For instance, consider a monoidal category $\cc$ such that $1$ is a left $\otimes$-generator. An algebra $A$ in $\cc$ is separable if and only if the forgetful functor $F:\cc_A \to \cc$ is a separable functor, see \cite[Proposition 4.3]{BT15}. 

The dual notion of separable functor, which is called naturally full functor, was introduced in \cite{ACMM06}. Semiseparable functors, introduced in \cite{AB22}, play the role of bridge between separable functors and naturally full functors. Both separable functors and naturally full functors are semiseparable. More precisely, a functor is separable (resp., naturally full) if and only if it is faithful (resp., full) and semiseparable, see \cite[Proposition 1.3]{AB22}. Besides, a functor is separable if and only if it is semiseparable and Maschke, if and only if it is semiseparable and dual Maschke, if and only if it is semiseparable and conservative, see \cite[Corollary 1.9]{AB22}.

Multifusion categories are a class of significant abelian monoidal categories. A fusion category is a multifusion category such that the unit object is simple, see \cite{Eti15}, \cite{ENO05}. Roughly speaking, from the definition, we know that the simplicity of the unit object controls some information of the category from an overall perspective. For example, the component subcategory decomposition $\cc = \oplus_{i,j \in I} \cc_{ij}$ is trivial i.e. the index set $I$ contains only one element if and only if the unit object is simple, see \cite[Remark 4.3.4, Theorem 4.3.8]{Eti15}.

In this paper, the main result shows that separability can provide information directly at the level of the category, by means of the unit object. More specifically, for a semisimple multiring category $\cc$ with left duals, $\cc$ is a ring category if and only if for any non-zero algebra $A$ in $\cc$ the functor $-\otimes A:\cc \to \cc_A$ is separable. We also show that, in a semisimple abelian monoidal category $\cc$ with biexact tensor product, the functor $-\otimes A:\cc \to \cc_A$ is always semiseparable, see Proposition \ref{semisimple semiseparable}. Therefore, by using semiseparability, our result can be further extended to the faithful, Maschke, dual Maschke and conservative cases, see Proposition \ref{semisimple multiring main result}. This leads to a list of equivalent definitions of fusion category, see Proposition \ref{main result}. 

From another point of view, for a multiring category with left duals, if it is not a ring category, then the separability of the functor $-\otimes A:\cc \to \cc_A$ depends on the algebra $A$. This motivates us to study the structure of algebras. By the component subcategory decomposition of a multiring category, see \cite[Remark 4.3.4]{Eti15}, we know that $\cc = \oplus_{i,j \in I} \cc_{ij}$. For any non-empty subset $J$ of $I$, we denote by $\cc_J :=\oplus_{i,j \in J} \cc_{ij}$ the full subcategory of $\cc$ and $1_J := \oplus_{i\in J} 1_i$. We show that if $A$ is an algebra in $\cc$, then $A_J:=1_J \otimes A \otimes 1_J$ is an algebra in $\cc_J$, see Proposition \ref{component algebra}. Thus, $A_J$ is an algebra in $\cc$, see Corollary \ref{AJ is algebra in C}. Furthermore, any non-zero algebra $A$ in $\cc$ must be of the form $A_J$ where $J=\{ i\in I| A_{ii} \not= 0 \}$, see Proposition \ref{equal zero}. Although $-\otimes A: \cc \to \cc_A$ may be not separable, we show that $-\otimes A_J: \cc_J \to (\cc_J)_{A_J}$ is separable, where $J=\{ i\in I| A_{ii} \not= 0 \}$, see Proposition \ref{restriction}. 

Since the main result provides equivalent characterizations of the fundamental definition, one can obtain numerous applications from it. In this paper, we focus on the following representative applications: to weak Hopf algebras (Subsection \ref{Application to weak Hopf algebras}), transfer of simplicity between the unit objects (Subsection \ref{Application to transfer of simplicity}), semisimple indecomposable module category (Subsection \ref{Application to semisimple indecomposable module category}) and Grothendieck ring (Subsection \ref{Application to Grothendieck ring}). 

The paper is organized as follows. In Section \ref{Preliminaries}, we recall some basic definitions and useful results. In Section \ref{Semiseparability in a semisimple abelian monoidal category}, we study the separability (Proposition \ref{semisimple induction separable}) and semiseparability (Proposition \ref{semisimple semiseparable}) of the functor $-\otimes A:\cc \to \cc_A$ in a semisimple abelian monoidal category $\cc$. In Section \ref{Multifusion category}, we show that in a semisimple multiring category $\cc$ with left duals, the unit object is simple if and only if for any non-zero algebra $A$ in $\cc$ the functor $-\otimes A:\cc \to \cc_A$ is separable, which is also equivalent to several other characterizations, see Proposition \ref{semisimple multiring main result}. Furthermore, the list of equivalent definitions is given by Proposition \ref{main result}. In Section \ref{Algebras in multiring category}, we study the structure of algebras in a multiring category with left duals. In Section \ref{Applications and examples}, we provide some applications and examples.

\emph{Notations.} In this paper, we always assume the monoidal categories are strict. A functor $F:\cc\to \dd$ means a covariant functor. For a monoidal category $\cc$, we write $\otimes$ as the tensor product and $1$ as the unit object. The subcategories of algebras and coalgebras in $\cc$ are denoted by $\mathrm{Alg}(\cc)$ and $\mathrm{Coalg}(\cc)$, respectively. For an algebra $A$ in a monoidal category $\cc$, we denote its multiplication and unit by $m_A$ and $u_A$, and by ${}_A\cc$ (resp., $\cc_A$) the category of left (resp., right) modules over $A$. For a coalgebra $C$ in a monoidal category $\cc$, we denote by $\Delta_C$ and $\varepsilon_C$ its comultiplication and counit, and by ${}^C\cc$ (resp., $\cc^C$) the category of left (resp., right) comodules over $C$. Furthermore, we let $\Bbbk$ be an algebraically closed field.

\section{Preliminaries}\label{Preliminaries}

Let $F: \cc \rightarrow \dd$ be a functor and consider the natural transformation
\begin{equation*}\label{nat_transf}
	\f : \Hom_{\cc}(-,-)\rightarrow \Hom_{\dd}(F-, F-),
\end{equation*}
defined by setting $\f_{X,X'}(f)= F(f)$, for any $f:X\rightarrow X'$ in $\cc$. If there is a natural transformation $\p : \Hom_{\dd}(F-, F-)\rightarrow \Hom_{\cc}(-,-)$ such that
\begin{itemize}
	\item $\p\circ\f = \id $, then $F$ is called \emph{separable} \cite{NVV89};
	\item $\f\circ\p = \id $, then $F$ is called \emph{naturally full} \cite{ACMM06};
	\item $\f\circ\p\circ\f = \f $, then $F$ is called \emph{semiseparable} \cite{AB22}.
\end{itemize}

In this paper, we are interested in the tensor functor by an algebra or a coalgebra. More precisely, let $\cc$ be a monoidal category. For an algebra $(A, m_A, u_A)$ in $\cc$, the tensor functor by $A$ is defined by $-\otimes A:\cc\to \cc_A$ by $M\mapsto (M\otimes A,M\otimes m_A)$, $f \mapsto f \otimes A$. Note that there is an adjoint pair $(-\otimes A,F)$ where $F:\cc_A\to\cc$ is the forgetful functor. The unit $\eta:\id\to F(-\otimes A)$ of the adjunction is given by $\eta_{M}= M\otimes u_A:M\to M\otimes A$, for every object $M$ in $\cc$, while the counit $\epsilon : F(-)\otimes A \to \id$ is given by $\epsilon_{N}= \mu^A_N: F(N)\otimes A= N\otimes A\to N$, for every object $(N, \mu^A_N)$ in $\cc_A$. Similarly, for a coalgebra $C$ in $\cc$, the tensor functor $-\otimes C:\cc\to\cc^C$, which is defined by $ M\mapsto (M\otimes C, M\otimes \Delta)$ and $f\mapsto f\otimes C$, is a right adjoint of the forgetful functor $F:\cc^C\to\cc$. The unit $\eta$ is defined for every $(N,\rho_N)\in\cc^C$ by $\eta_{(N,\rho_N)}=\rho_N:N\to N\otimes C=F(N)\otimes C$, while the counit $\epsilon$ is given for any $M\in\cc$ by $\epsilon_M= M\otimes\varepsilon :F(M\otimes C)=M\otimes C\to M$, see e.g. \cite[Section 2]{MT21}. Recall that the semiseparability of $-\otimes A:\cc\to \cc_A$ and $-\otimes C:\cc\to\cc^C$ are characterized in the following way.

\begin{proposition}\cite[Proposition 3.20]{BZ26}\label{induction functor} 
Let $\cc$ be a monoidal category, and $A$ be an algebra in $\cc$. Then, the following assertions are equivalent.
\begin{itemize}
\item[$(i)$] The unit $u_A:1\to A$ is regular (resp., split-mono, split-epi) as a morphism in $\cc$.

\item[$(ii)$] the tensor functor $-\otimes A: \cc \to \cc_A$ is semiseparable (resp., separable, naturally full).

\item[$(iii)$] the tensor functor $A\otimes -: \cc \to _A\!\cc$ is semiseparable (resp., separable, naturally full).
\end{itemize} 
\end{proposition}

\begin{proposition}\cite[Proposition 4.16]{BZ26}\label{coinduction functor}
Let $\cc$ be a monoidal category and $C$ be a coalgebra in $\cc$. Then, the following assertions are equivalent.
\begin{enumerate}
\item[$(i)$] The counit $\varepsilon_{C}:C\to 1$ is regular (resp., split-epi, split-mono) as a morphism in $\cc$.

\item[$(ii)$] the tensor functor $-\otimes C:\cc\to \cc^C$ is semiseparable (resp., separable, naturally full).

\item[$(iii)$] the tensor functor $C\otimes -:\cc\to ^C\!\!\cc$ is semiseparable (resp., separable, naturally full).
\end{enumerate} 
\end{proposition}

\begin{definition}\cite[Definition 1.5.1]{Eti15}
Let $\cc$ be an abelian category. A nonzero object $X$ in $\cc$ is called simple if $0$ and $X$ are its only subobjects. An object $X$ in $\cc$ is called semisimple if it is a direct sum of simple objects, and $\cc$ is called semisimple if every object of $\cc$ is semisimple.    
\end{definition}

In this paper, an abelian monoidal category is a monoidal category which is abelian with additive tensor product.

\begin{definition}\cite[Definition 4.1.1, Definition 4.2.3]{Eti15}
\begin{enumerate}
\item A multiring category $\cc$ over $\Bbbk$ is a locally finite $\Bbbk$-linear abelian monoidal category with bilinear and biexact tensor product. If in addition $\mathrm{End}_{\cc}(1) = \Bbbk$, we will call $\cc$ a ring category.
 
\item A multitensor category $\cc$ over $\Bbbk$ is a locally finite $\Bbbk$-linear abelian rigid monoidal category with bilinear tensor product. If in addition $\mathrm{End}_{\cc}(1) = \Bbbk$, and $\cc$ is indecomposable i.e. $\cc$ is not equivalent to a direct sum of nonzero multitensor categories, then we will call $\cc$ a tensor category.

\item A multifusion category is a finite semisimple multitensor category. A fusion
category is a multifusion category with $\mathrm{End}_{\cc}(1) = \Bbbk$, i.e., a finite semisimple tensor
category.
\end{enumerate}
\end{definition}

In fact, the tensor product of a multitensor category is always biexact, see \cite[Proposition 4.2.1]{Eti15}. This implies a multitensor category is always a multiring category.

\section{Semiseparability in a semisimple abelian monoidal category}\label{Semiseparability in a semisimple abelian monoidal category}

In this section, we study the separability (Proposition \ref{semisimple induction separable}) and semiseparability (Proposition \ref{semisimple semiseparable}) of the functor $-\otimes A: \cc \to \cc_A$ for any algebra $A$ in $\cc$, where $\cc$ is a semisimple abelian monoidal category with biexact tensor product. First, recall that a morphism $f$ is called regular if there is a morphism $g$ such that $f\circ g\circ f=f$.

\begin{lemma}\label{semisimple regular}
An abelian monoidal category $\cc$ is semisimple if and only if every morphism in $\cc$ is regular.
\end{lemma}

\begin{proof}
$\Rightarrow$ For any morphism $f: X \to Y$ in $\cc$, consider the image factorization of $f=\varphi \circ \psi$, where $\psi:X \to \im (f)$ is an epimorphism and $\varphi:\im (f) \to Y$ is a monomorphism. Since $\cc$ is semisimple, there are morphisms $\psi':\im (f) \to X$ and $\varphi':Y \to \im (f)$ such that $\psi \circ \psi' = \id_{\im (f)} = \varphi' \circ \varphi$. Thus, $f \psi' \varphi' f = \varphi \psi \psi' \varphi' \varphi \psi = \varphi \psi = f$, i.e. $f$ is regular. 

$\Leftarrow$ Suppose every morphism in $\cc$ is regular. We have every monomorphism is split-mono and every epimorphism is split-epi. Hence, $\cc$ is semisimple.
\end{proof}

\begin{remark}
In the category $\Set$ of sets, the requirement that every morphism is regular is equivalent to the Axiom of choice, see \cite[Example 3.9]{AB22}. However, $\Set$ is not an abelian category. 
\end{remark}

\begin{proposition}\label{semisimple induction separable}
Let $\cc$ be a semisimple abelian monoidal category with biexact tensor product such that the unit object is simple. Then, 
\begin{enumerate}[$(1)$]
\item for any non-zero algebra $A$ in $\cc$, the tensor functors $-\otimes A:\cc \to \cc_A$ and $A\otimes -: \cc \to _A\!\cc$ are separable; 

\item for any non-zero coalgebra $C$ in $\cc$, the tensor functors $-\otimes C:\cc\to \cc^C$ and $C\otimes -:\cc\to ^C\!\!\cc$ are separable. 
\end{enumerate}
\end{proposition}

\begin{proof}
$(1)$ Since the unit object $1$ is simple, $u_A: 1 \to A$ is a monomorphism. Note that $\cc$ is semisimple, which means $u_A$ is split-mono. By Proposition \ref{induction functor}, $-\otimes A$ and $A\otimes -$ are separable. $(2)$ It follows similarly to $(1)$, by Proposition \ref{coinduction functor}.
\end{proof}

\begin{proposition}\label{semisimple semiseparable}
Let $\cc$ be a semisimple abelian monoidal category with biexact tensor product. Then, 
\begin{enumerate}[$(1)$]
    \item for any algebra $A$ in $\cc$, the tensor functors $-\otimes A: \cc \to \cc_A$ and $A\otimes -:\cc \to _A\!\cc$ are semiseparable; 
    
    \item for any coalgebra $C$ in $\cc$, the tensor functors $-\otimes C:\cc\to \cc^C$ and $C\otimes -:\cc\to ^C\!\!\cc$ are semiseparable. 
\end{enumerate}
\end{proposition}

\begin{proof}
$(1)$ By Lemma \ref{semisimple regular}, $u_A$ is regular. By Proposition \ref{induction functor}, $-\otimes A$ and $A\otimes -$ are semiseparable. $(2)$ It follows similarly to $(1)$, by Proposition \ref{coinduction functor}.
\end{proof}

Our next goal is to study the converse of Proposition \ref{semisimple induction separable}. More specifically, suppose the tensor functors are separable for any non-zero (co)algebra, we want to know if the unit object is simple. 

\begin{lemma}\label{lemma iso}
Let $\cc$ be a semisimple abelian monoidal category such that the unit object $1$ has finite length. Let $X$ be a non-zero subobject of $1$. Suppose either of the following conditions is satisfied:
\begin{enumerate}[$(1)$]
    \item $X$ is an algebra in $\cc$, and the tensor functor $-\otimes A:\cc \to \cc_A$ is separable for any non-zero algebra $A$ in $\cc$;
    
    \item $X$ is a coalgebra in $\cc$, and the tensor functor $-\otimes C:\cc\to \cc^C$ is separable for any non-zero coalgebra $C$ in $\cc$.
\end{enumerate}
Then, $X$ is isomorphic to $1$.
\end{lemma}

\begin{proof}
$(1)$ Because $\cc$ is semisimple and $1$ has finite length, $1$ is the direct sum of finitely many simple objects. Since $X$ is a subobject of $1$, $X$ is isomorphic to the direct sum of simple subobjects of $1$ whose multiplicity is not greater than the multiplicity in the direct sum decomposition of $1$. By Proposition \ref{induction functor}, $u_X: 1 \to X$ is split-mono, which means $1$ is a subobject of $X$. Therefore, the multiplicities of simple objects in the direct sum decompositions of $1$ and $X$ are exactly the same. As a consequence, $u_X: 1 \to X$ is an isomorphism. $(2)$ It follows similarly to $(1)$, by Proposition \ref{coinduction functor}.
\end{proof}

Using this lemma and combining with the previous propositions, we obtain following results directly.

\begin{proposition}
Let $\cc$ be a semisimple abelian monoidal category such that the unit object $1$ has finite length. Consider the following assertions:
\begin{enumerate}[$(1)$]
\item the unit object is simple.

\item for any non-zero algebra $A$ in $\cc$, the tensor functor $-\otimes A:\cc \to \cc_A$ is separable.

\item the unit object is simple in $\mathrm{Alg}(\cc)$.
\end{enumerate} 
Then, $(1)$ implies $(2)$, and $(2)$ implies $(3)$.
\end{proposition}

\begin{proof}
$(1)\Rightarrow (2)$: It is Proposition \ref{semisimple induction separable}. $(2) \Rightarrow (3)$ Suppose $X$ is a non-zero subobject of $1$ in $\mathrm{Alg}(\cc)$. By lemma \ref{lemma iso}, $X$ is isomorphic to $1$. Hence, the unit object is simple in $\mathrm{Alg}(\cc)$.
\end{proof}

Similarly, we can consider the coalgebra case. 

\begin{proposition}
Let $\cc$ be a semisimple abelian monoidal category such that the unit object $1$ has finite length. Consider the following assertions:
\begin{enumerate}[$(1)$]
\item the unit object is simple.

\item for any non-zero coalgebra $C$ in $\cc$, the tensor functor $-\otimes C:\cc\to \cc^C$ is separable.

\item the unit object is simple in $\mathrm{Coalg}(\cc)$.
\end{enumerate} 
Then, $(1)$ implies $(2)$, and $(2)$ implies $(3)$.
\end{proposition}

Our motivation is finding the enviroment such that $(2)$ implies $(1)$. Now we know that $(2)$ implies $(3)$. This means that we must focus on studying the subobjects of $1$. Is there an enviroment such that all the subobjects of $1$ are both algebras and coalgebras? In the next section, we turn our attention to multifusion category.

\section{Multifusion category}\label{Multifusion category}

Let $\cc$ be a multiring category with left duals. By \cite[Theorem 4.3.8 (ii)]{Eti15}, the unit object $1$ of $\cc$ is semisimple,
and it is a direct sum of pairwise non-isomorphic simple objects $1_i$. In fact, let $\{ p_i\}_{i \in I}$ be the primitive idempotents of the algebra $\mathrm{End}(1)$, $1_i$ is the image of $p_i$. In the following, we write $1 = \oplus_{i\in I} 1_i$, and denote the canonical injection and projection by $i_i:1_i \to 1$ and $p_i:1 \to 1_i$, respectively. Note also that $1_i \otimes 1_j =0$ for $i \not= j$, see \cite[Exercise 4.3.3]{Eti15}. Besides, by \cite[Remark 4.3.4]{Eti15}, there is a decomposition $\cc = \oplus_{i,j \in I} \cc_{ij}$ where $\cc_{ij} = 1_i \otimes \cc \otimes 1_j$, and the categories $\cc_{ii}$ are ring categories with unit object $1_i$.

\begin{lemma}\label{algebra}
For $i \in I$, $1_i$ is both an algebra and a coalgebra in $\cc$.
\end{lemma}

\begin{proof}
Denote the canonical isomorphisms by $\tau:1_i \otimes (\oplus_{j\in I} 1_j)\to \oplus_{j\in I} (1_i \otimes 1_j)$ and $\sigma: \oplus_{j\in I} (1_j \otimes 1_i) \to (\oplus_{j\in I} 1_j) \otimes 1_i$, and the canonical projections by $p^r_{j}: \oplus_{j\in I} (1_i \otimes 1_{j}) \to 1_i\otimes 1_j$ and $p^l_{j}: \oplus_{j\in I} (1_{j} \otimes 1_i) \to 1_j\otimes 1_i$ for any $j$ in $I$. Since  $1_i \otimes 1_j =0$ for $i \not= j$, we have $p^r_{i}:\oplus_{j\in I} (1_i \otimes 1_{j}) \to 1_i\otimes 1_i$ and $p^l_{i}:\oplus_{j\in I} (1_{j} \otimes 1_i) \to 1_i\otimes 1_i$ are isomorphisms. Indeed, one can observe that $p_j^l$ and $p_j^r$ are always zero for $j \not= i$. Furthermore, because $p^r_{i} \tau = \id_{1_i}\otimes p_i$ and $p^l_{i} = (p_i \otimes \id_{1_i})\sigma$, we obtain that $\id_{1_i} \otimes p_i = p^r_{i} \tau = p^l_{i} (p^l_{i})^{-1} p^r_{i} \tau = (p_j\otimes \id_{1_i})\sigma (p^l_{i})^{-1} p^r_{i} \tau$.

For $i \not= j$, since $1_i \otimes 1_j =0$, we have
$i_i \otimes i_j = 0 = i_j \otimes i_i$. Hence, 
$$
\begin{aligned}
i_i\otimes \id_1 &= i_i \otimes \sum_{j\in I} (i_j p_j) = i_i \otimes i_i p_i = (i_i \otimes i_i)(\id_{1_i}\otimes p_i)\\ 
&= (i_i \otimes i_i) (p_i\otimes \id_{1_i})\sigma (p^l_{i})^{-1} p^r_{i} \tau = (i_i p_i \otimes i_i)\sigma (p^l_{i})^{-1} p^r_{i} \tau\\ 
&= (\sum_{j\in I} (i_j p_j)\otimes i_i) \sigma (p^l_{i})^{-1} p^r_{i} \tau = (\id_1 \otimes i_i)\sigma (p^l_{i})^{-1} p^r_{i} \tau.
\end{aligned}
$$ 
Thus, by the naturality of unit constraints, we have $i_i r_{1_i} = r_1 (i_i \otimes \id_1) = l_1 (\id_1 \otimes i_i)\sigma (p^l_{i})^{-1} p^r_{i} \tau = i_i l_{1_i}\sigma (p^l_{i})^{-1} p^r_{i} \tau$, which implies $r_{1_i} = l_{1_i}\sigma (p^l_{i})^{-1} p^r_{i} \tau$. 

We claim that $1_i$ is an algebra in $\cc$ with the multiplication and the unit defined by $m_{1_i}:= r_{1_i} \tau^{-1} (p^r_{i})^{-1} = l_{1_i} \sigma (p^l_{i})^{-1}: 1_i \otimes 1_i \to 1_i$ and $u_{1_i}:= p_i: 1 \to 1_i$. First, we want to check that $m_{1_i} (\id_{1_i} \otimes m_{1_i}) = m_{1_i} (m_{1_i} \otimes \id_{1_i})$. Note that $m_{1_i}$ is an isomorphism, it suffices to show $\id_{1_i} \otimes m_{1_i} = m_{1_i} \otimes \id_{1_i}$. Because we are under the strict assumption, it is clear that $\id_{1_i}\otimes l_{1_i} = r_{1_i} \otimes \id_{1_i}$. Since $\tau^{-1} (p^r_{i})^{-1} = (\id_{1_i} \otimes p_i)^{-1} = \id_{1_i} \otimes i_i$ and $\sigma (p^l_{i})^{-1} = (p_i \otimes \id_{1_i})^{-1} = i_i \otimes \id_{1_i}$, we have 
$$
\begin{aligned}
\id_{1_i} \otimes m_{1_i} &= \id_{1_i} \otimes l_{1_i} \sigma (p^l_{i})^{-1} = (\id_{1_i} \otimes l_{1_i}) (\id_{1_i} \otimes i_i \otimes \id_{1_i})\\ 
&= (r_{1_i} \otimes \id_{1_i})(\tau^{-1} (p^r_{i})^{-1} \otimes \id_{1_i}) = r_{1_i} \tau^{-1} (p^r_{i})^{-1} \otimes \id_{1_i} = m_{1_i} \otimes \id_{1_i}.
\end{aligned} 
$$
Moreover, $m_{1_i} (\id_{1_i} \otimes u_{1_i}) = r_{1_i} \tau^{-1} (p^r_{i})^{-1} (\id_{1_i} \otimes p_i) = r_{1_i} \tau^{-1} \tau = r_{1_i}$, and similarly $m_{1_i} (u_{1_i} \otimes \id_{1_i}) = l_{1_i} \sigma (p^l_{i})^{-1} (p_i \otimes \id_{1_i}) = l_{1_i} \sigma \sigma^{-1} = l_{1_i}$. Thus, $(1_i, m_{1_i}, u_{1_i})$ is an algebra in $\cc$.

Similarly, we claim that $1_i$ is a coalgebra in $\cc$ with the comultiplication and the counit defined by $\Delta_{1_i}:= p_i^r \tau r_{1_i}^{-1} = p_i^l \sigma^{-1} l_{1_i}^{-1}: 1_i \to 1_i \otimes 1_i$ and $\varepsilon_{1_i}:= i_i: 1_i \to 1$. In fact, $\Delta_{1_i} = m_{1_i}^{-1}$, so it is automatically that $(\Delta_{1_i} \otimes \id_{1_i})\Delta_{1_i} = (\id_{1_i} \otimes \Delta_{1_i})\Delta_{1_i}$. In addition, $(\id_{1_i} \otimes \varepsilon_{1_i}) \Delta_{1_i} = (\id_{1_i} \otimes i_i) p_i^r \tau r_{1_i}^{-1} = \tau^{-1} (p_i^r)^{-1} p_i^r \tau r_{1_i}^{-1} = r_{1_i}^{-1}$ and $(\varepsilon_{1_i} \otimes \id_{1_i})\Delta_{1_i} = (i_i \otimes \id_{1_i}) p_i^l \sigma^{-1} l_{1_i}^{-1} = \sigma (p_i^l)^{-1} p_i^l \sigma^{-1} l_{1_i}^{-1} = l_{1_i}^{-1}$. Thus, it is indeed a coalgebra in $\cc$.
\end{proof}

\begin{lemma}\label{multiring}
Let $\cc$ be a multiring category with left duals. If either of the following conditions is satisfied, 
\begin{enumerate}[$(1)$]
\item for any non-zero algebra $A$ in $\cc$, the tensor functor $-\otimes A:\cc \to \cc_A$ is separable,

\item for any non-zero coalgebra $C$ in $\cc$, the tensor functor $-\otimes C:\cc\to \cc^C$ is separable, 
\end{enumerate}
then the unit object of $\cc$ is simple.
\end{lemma}

\begin{proof}
$(1)$ Suppose $1$ has a non-zero simple subobject $1_i$. By Lemma \ref{algebra}, $1_i$ is an algebra in $\cc$. By the given assumption and Proposition \ref{induction functor}, $-\otimes 1_i$ is separable, and hence $u_{1_i}$ is split-mono. Note that $u_{1_i} = p_i: 1 \to 1_i$ is a split epimorphism. This implies $1$ is isomorphic to $1_i$. $(2)$ It follows in a similar manner to $(1)$.    
\end{proof}

Now, we can consider the main theorem in this paper. Recall that a functor $F:\cc \to \dd$ is called a Maschke functor if it reflects split-monomorphisms, i.e. for every morphism $i$ in $\cc$ such that $Fi$ is split-mono, then $i$ is split-mono. Similarly, $F$ is a dual Maschke functor if it reflects split-epimorphisms. A functor is called conservative if it reflects isomorphisms. 

\begin{proposition}\label{semisimple multiring main result}
Let $\cc$ be a semisimple multiring category with left duals. Then, the following are equivalent:  

\begin{enumerate}[$(1)$]
\item the unit object is simple, i.e. $\cc$ is a ring category.

\item for any non-zero algebra $A$ in $\cc$, the functor $-\otimes A:\cc \to \cc_A$ is separable.

\item for any non-zero coalgebra $C$ in $\cc$, the functor $-\otimes C:\cc\to \cc^C$ is separable. 

\item for any non-zero algebra $A$ in $\cc$, the functor $-\otimes A:\cc \to \cc_A$ is faithful.

\item for any non-zero coalgebra $C$ in $\cc$, the functor $-\otimes C:\cc\to \cc^C$ is faithful.

\item for any non-zero algebra $A$ in $\cc$, the functor $-\otimes A:\cc \to \cc_A$ is Maschke.

\item for any non-zero coalgebra $C$ in $\cc$, the functor $-\otimes C:\cc\to \cc^C$ is Maschke.

\item for any non-zero algebra $A$ in $\cc$, the functor $-\otimes A:\cc \to \cc_A$ is dual Maschke.

\item for any non-zero coalgebra $C$ in $\cc$, the functor $-\otimes C:\cc\to \cc^C$ is dual Maschke.

\item for any non-zero algebra $A$ in $\cc$, the functor $-\otimes A:\cc \to \cc_A$ is conservative.

\item for any non-zero coalgebra $C$ in $\cc$, the functor $-\otimes C:\cc\to \cc^C$ is conservative.

\item for any non-zero algebra $A$ in $\cc$, its unit $u_A:1 \to A$ is a monomorphism.

\item for any non-zero coalgebra $C$ in $\cc$, its counit $\varepsilon_C:C \to 1$ is an epimorphism.

\item any non-zero algebra morphism $f:1 \to A$ in $\cc$ is a monomorphism.

\item any non-zero coalgebra morphism $f:C \to 1$ in $\cc$ is an epimorphism.
\end{enumerate}
\end{proposition}

\begin{proof}
$(1) \Rightarrow (2)$ and $(1) \Rightarrow (3)$ follow from Proposition \ref{semisimple induction separable}. $(2) \Rightarrow (1)$ and $(3) \Rightarrow (1)$ follow from Lemma \ref{multiring}. Hence, $(1) \Leftrightarrow (2) \Leftrightarrow (3)$

By Proposition \ref{semisimple semiseparable}, the tensor functors $-\otimes A$ and $- \otimes C$ are always semiseparable in a multifusion category. Recall that a functor is separable if and only if it is semiseparable and faithful, see \cite[Proposition 1.3]{AB22}. Thus, $(2) \Leftrightarrow (4)$ and $(3) \Leftrightarrow (5)$. 

Besides, by \cite[Corollary 1.9]{AB22}, a functor is separable if and only if it is semiseparable and Maschke, if and only if it is semiseparable and dual Maschke, if and only if it is semiseparable and conservative. Hence, $(2) \Leftrightarrow (6) \Leftrightarrow (8) \Leftrightarrow (10)$ and $(3) \Leftrightarrow (7) \Leftrightarrow (9) \Leftrightarrow (11)$.

$(1)\Rightarrow (12)$ Because $\ker(u_A)$ is a subobject of $1$, $1$ being simple and $u_A$ being non-zero mean that $\ker(u_A) = 0$. One can show $(1)\Rightarrow (13), (14), (15)$ similarly. $(12)\Rightarrow (1)$ and $(13) \Rightarrow (1)$ follow from Lemma \ref{algebra}, since $u_{1_i}: 1 \to 1_i$ being a monomorphism or $\varepsilon_{1_i}: 1_i \to 1$ being an epimorphism would imply $1 = 1_i$. $(14) \Rightarrow (12)$ and $(15) \Rightarrow (13)$ are obvious. Hence, $(1) \Leftrightarrow (12) \Leftrightarrow (13) \Leftrightarrow (14) \Leftrightarrow (15)$.
\end{proof}

By applying Proposition \ref{semisimple multiring main result} to multifusion category, we obtain the equivalent definitions of fusion category.

\begin{proposition}\label{main result}
Let $\cc$ be a multifusion category. Then, the following are equivalent:  

\begin{enumerate}[$(1)$]
\item the unit object is simple, i.e. $\cc$ is a fusion category.

\item for any non-zero algebra $A$ in $\cc$, the functor $-\otimes A:\cc \to \cc_A$ is separable.

\item for any non-zero coalgebra $C$ in $\cc$, the functor $-\otimes C:\cc\to \cc^C$ is separable. 

\item for any non-zero algebra $A$ in $\cc$, the functor $-\otimes A:\cc \to \cc_A$ is faithful.

\item for any non-zero coalgebra $C$ in $\cc$, the functor $-\otimes C:\cc\to \cc^C$ is faithful.

\item for any non-zero algebra $A$ in $\cc$, the functor $-\otimes A:\cc \to \cc_A$ is Maschke.

\item for any non-zero coalgebra $C$ in $\cc$, the functor $-\otimes C:\cc\to \cc^C$ is Maschke.

\item for any non-zero algebra $A$ in $\cc$, the functor $-\otimes A:\cc \to \cc_A$ is dual Maschke.

\item for any non-zero coalgebra $C$ in $\cc$, the functor $-\otimes C:\cc\to \cc^C$ is dual Maschke.

\item for any non-zero algebra $A$ in $\cc$, the functor $-\otimes A:\cc \to \cc_A$ is conservative.

\item for any non-zero coalgebra $C$ in $\cc$, the functor $-\otimes C:\cc\to \cc^C$ is conservative.

\item for any non-zero algebra $A$ in $\cc$, its unit $u_A:1 \to A$ is a monomorphism.

\item for any non-zero coalgebra $C$ in $\cc$, its counit $\varepsilon_C:C \to 1$ is an epimorphism.

\item any non-zero algebra morphism $f:1 \to A$ in $\cc$ is a monomorphism.

\item any non-zero coalgebra morphism $f:C \to 1$ in $\cc$ is an epimorphism.
\end{enumerate}
\end{proposition}

\begin{remark}
Let $\cc$ be a semisimple multiring category with left duals. Recall that there is an associated idempotent natural transformation $e:\id_{\cc} \to \id_{\cc}$ for any semiseparable functor $F: \cc \to \dd$, see \cite[Definition 1.5]{AB22}. By Proposition \ref{semisimple semiseparable}, the tensor functor $-\otimes A: \cc \to \cc_A$, for an algebra $A$ in $\cc$, is semiseparable. As observed in \cite[page 730]{BT15}, the induction functor $u_A^*=-\otimes _1 A: \cc_1\to \cc_A$ is just the tensor functor $-\otimes A:\cc\to\cc_A$. Hence, by \cite[Remark 3.6]{BZ26}, the associated idempotent natural transformation of $-\otimes A:\cc\to\cc_A$ is given by 
$$
e_M=\nu_M\circ\eta_M=\nu^1_M (\id_M\otimes \psi' \varphi'u_A) = r_M^1 (\id_M\otimes \psi' \varphi' \varphi \psi) = \id_M\otimes \psi' \psi,
$$ 
where $u_A = \varphi \psi$ is the image factorization and $\psi', \varphi'$ are morphisms such that $\psi \circ \psi' = \id_{\im (f)} = \varphi' \circ \varphi$ as in the proof of Lemma \ref{semisimple regular}. By \cite[Corollary 1.7]{AB22}, $-\otimes A:\cc\to\cc_A$ is separable if and only if $e = \id$. Note that $e = \id$ if and only if $\psi' \psi = \id_1$, if and only if $\psi:1 \to \im(u_A)$ is an isomorphism. Now, suppose the unit of $\cc$ is not simple. Since $1 = \oplus_{i\in I} 1_i$, we know that $1_i$ is not isomorphic to $1$. By Lemma \ref{algebra}, $1_i$ is an algebra in $\cc$ with unit $u_{1_i} = p_i:1 \to 1_i$. Because $\im(u_{1_i})$ is a subobject of $1_i$, $\psi: 1\to \im(u_{1_i})$ is not an isomorphism. As a result, $-\otimes 1_i$ is not separable. This observation supports Proposition \ref{semisimple multiring main result}.
\end{remark}

\begin{remark}
One may ask what happens if for any non-zero algebra $A$ in $\cc$, the functor $-\otimes A:\cc \to \cc_A$ is naturally full. By Proposition \ref{induction functor}, this is equivalent to say $u_A$ is split-epi for any non-zero algebra $A$ in $\cc$. This is impossible if $\cc$ is a non-zero abelian monoidal category. In fact, $u_{1 \oplus 1} = \left( \begin{smallmatrix} \id_1 \\ \id_1 \end{smallmatrix} \right):1 \to 1\oplus 1$ is not an epimorphism, since $\left( \begin{smallmatrix} \id_1 & -\id_1 \end{smallmatrix}\right) \left( \begin{smallmatrix} \id_1 \\ \id_1 \end{smallmatrix} \right)$ while $\left( \begin{smallmatrix} \id_1 & -\id_1 \end{smallmatrix}\right) \not=0$. Besides, if $\cc$ is a multiring category with left duals, $u_A: 1 \to A$ being epimorphism means that $A$ is isomorphic to $\oplus_{i \in J} 1_i$ for some subset $J$ of $I$.  
\end{remark}

\begin{remark}\label{forgetful functor separable}
Denote the forgetful functor by $F:\cc_A \to \cc$. We are curious about the relation between the separability of $-\otimes A:\cc \to \cc_A$ and that of $F \circ (-\otimes A):\cc \to \cc$. By \cite[Proposition 46]{CMZ02}, if $F \circ (-\otimes A)$ is separable, then $-\otimes A:\cc \to \cc_A$ is separable. Conversely, if both $F$ and $-\otimes A$ are separable, then $F \circ (-\otimes A)$ is separable. Thus, the question becomes: under what conditions is $F$ separable? An answer is provided by \cite[Proposition 4.3]{BT15}: Let $\cc$ be a monoidal category such that $1$ is a left $\otimes$-generator for $\cc$. If $A$ is an algebra in $\cc$ then $A$ is separable if and only if the forgetful functor $F:\cc_A \to \cc$ is a separable functor.
\end{remark}

\section{Algebras in multiring category}\label{Algebras in multiring category}

Let $\cc$ be a multiring category with left duals. From Lemma \ref{multiring}, we know that the simplicity of the unit object in $\cc$ is determined by the property of algebras in $\cc$, while the Proposition \ref{semisimple multiring main result} implies this property is also determined by the simplicity of the unit if $\cc$ is semisimple. This motivates us to further explore the structure of algebras in $\cc$. In this section, we will consider the relation between algebras in $\cc$ and algebras in the component subcategory $\cc_{ii}$. 

Because $\cc = \oplus_{i,j \in I} \cc_{ij}$, for the sake of brevity, we write $X = \oplus_{i,j \in I} X_{ij}$ for any object $X$ where $X_{ij} = 1_i \otimes X \otimes 1_j$ is in $\cc_{ij}$. Denote the canonical projection and injection by $p_{X_{ij}}:= p_i \otimes \id_X \otimes p_j:X \to X_{ij}$ and $i_{X_{ij}}:= i_i \otimes \id_X \otimes i_j: X_{ij} \to X$ respectively. 

\begin{lemma}\label{morphism between component categories}
Let $\cc$ be a multiring category with left duals. Then, there is only the zero morphism  from an object in $\cc_{ij}$ to an object in $\cc_{i'j'}$ for $i\not= i'$ or $j \not= j'$.
\end{lemma}

\begin{proof}
The case of $\cc_{ij}$ or $\cc_{i'j'}$ being zero is trivial. We only need to consider the case that they are both non-zero. Let $X_{ij}$, $X_{i'j'}$ be non-zero objects in $\cc_{ij}$, $\cc_{i'j'}$, respectively. Assume there is a morphism $f:X_{ij} \to X_{i'j'}$, then $1_s \otimes f \otimes 1_t = 0$ for all $s$, $t$ in $I$. 
Since there is an isomorphism $\tau: (\oplus_{s \in I} 1_s) \otimes X_{ij} \otimes (\oplus_{t \in I} 1_t) \to \oplus_{s,t \in I} (1_s \otimes X_{ij} \otimes 1_t)$, we can write $f = f \tau^{-1} \tau = \sum_{s,t \in I} f \tau^{-1} i_{st} p_{st} \tau$, where $i_{st}: 1_s \otimes X_{ij} \otimes 1_t \to \oplus_{s,t \in I} (1_s \otimes X_{ij} \otimes 1_t)$ and $p_{st}:\oplus_{s,t \in I} (1_s \otimes X_{ij} \otimes 1_t) \to 1_s \otimes X_{ij} \otimes 1_t$ are canonical injection and projection, respectively. Because $\tau^{-1} i_{st} = i_s \otimes \id_{X_{ij}} \otimes i_t$, we have $f \tau^{-1} i_{st} = (1 \otimes f \otimes 1) (i_s \otimes \id_{X_{ij}} \otimes i_t) = (i_s \otimes \id_{X_{ij}} \otimes i_t)(1_s \otimes f \otimes 1_t) = 0$. Thus, $f = \sum_{s,t \in I} f \tau^{-1} i_{st} p_{st} \tau = 0$.   
\end{proof}

According to Lemma \ref{morphism between component categories}, any morphism $f:X \to Y$ can be written as $f = \oplus_{i,j \in I} f_{ij}$, where $f_{ij} = \id_{1_i} \otimes f \otimes \id_{1_j}:X_{ij} \to Y_{ij}$.

\begin{invisible}

\begin{lemma}
Let $\cc$ be an abelian monoidal category with left duals. For objects $X_1$, $X_2$ in $\cc$, the left dual of $X_1 \oplus X_2$ is given by $(X_1 \oplus X_2)^*: = X_1^* \oplus X_2^*$.
\end{lemma}

\begin{proof}
Denote the evaluation and coevaluation of $X_1$ and $X_2$ by $\ev_1$, $\ev_2$, $\coev_1$, $\coev_2$, respectively. Define morphisms by
$$
\ev_{X_1 \oplus X_2}:= \ev_1 (p_{X_1^*} \otimes p_{X_1}) + \ev_2 (p_{X_2^*} \otimes p_{X_2}): (X_1^* \oplus X_2^*) \otimes (X_1 \oplus X_2) \to 1, 
$$
and
$$
\coev_{X_1 \oplus X_2}:= (i_{X_1} \otimes i_{X_1^*}) \coev_1 + (i_{X_2} \otimes i_{X_2^*}) \coev_2: 1 \to (X_1 \oplus X_2) \otimes (X_1^* \oplus X_2^*).
$$
Since 
\begin{align*}
&(\id_{X_1\oplus X_2} \otimes \ev_{X_1\oplus X_2}) (\coev_{X_1\oplus X_2} \otimes \id_{X_1\oplus X_2})\\
= &(\id_{X_1\oplus X_2} \otimes (\ev_1 (p_{X_1^*} \otimes p_{X_1}) + \ev_2 (p_{X_2^*} \otimes p_{X_2})))\\ 
&((i_{X_1} \otimes i_{X_1^*}) \coev_1 + (i_{X_2} \otimes i_{X_2^*}) \coev_2 \otimes \id_{X_1\oplus X_2})\\
= &(\id_{X_1\oplus X_2} \otimes \ev_1) (i_{X_1} \otimes \id_{X_1^*} \otimes p_{X_1})(\coev_1 \otimes \id_{X_1\oplus X_2})\\
&+(\id_{X_1\oplus X_2} \otimes \ev_2) (i_{X_2} \otimes \id_{X_2^*} \otimes p_{X_2})(\coev_2 \otimes \id_{X_1\oplus X_2})\\
= &i_{X_1}(\id_{X_1} \otimes \ev_1)(\coev_1 \otimes \id_{X_1})p_{X_1} + i_{X_2}(\id_{X_2} \otimes \ev_2)(\coev_2 \otimes \id_{X_2})p_{X_2}\\
= &i_{X_1}p_{X_1} + i_{X_2}p_{X_2} = \id_{X_1\oplus X_2}, 
\end{align*}
and similarly
$$
(\ev_{X_1\oplus X_2} \otimes \id_{X_1^* \oplus X_2^*})(\id_{X_1^* \oplus X_2^*} \otimes \coev_{X_1 \oplus X_2}) = \id_{X_1^* \oplus X_2^*},
$$
we obtain that $X_1^* \oplus X_2^*$ is a left dual of $X_1 \oplus X_2$.
\end{proof}
\end{invisible}

Consider the component subcategory decomposition $\cc = \oplus_{i,j \in I} \cc_{ij}$. For any non-empty subset $J$ of $I$, we denote by $\cc_J :=\oplus_{i,j \in J} \cc_{ij}$ the full subcategory of $\cc$ and $1_J := \oplus_{i\in J} 1_i$. We also denote by $p_J: 1 \to \oplus_{j \in J} 1_j$ and $i_J: \oplus_{j \in J} 1_j \to 1$ the canonical projection and inclusion.

\begin{proposition}\label{CJ of multiring}
Let $\cc$ be a multiring category with left duals. For any non-empty subset $J$ of $I$, the subcategory $\cc_J$ of $\cc$ is also a multiring category with left duals.
\end{proposition}

\begin{proof}
It is clear that $\cc_J$ is an additive category with zero object. For two objects $X = \oplus_{i,j \in J} X_{ij}$, $Y= \oplus_{i,j \in J} Y_{ij}$ in $\cc_J$, we know that $X_{ij} \oplus Y_{ij} = (1_i \otimes X \otimes 1_j) \oplus (1_i \otimes Y \otimes 1_j) \cong 1_i \otimes (X \oplus Y) \otimes 1_j$ is an object in $\cc_{ij}$. Therefore, $X \oplus Y = \oplus_{i,j \in J} (X_{ij} \oplus Y_{ij})$ is in $\cc_J$. Let $f:X\to Y$ be a morphism in $\cc_J$. Suppose $\ker(f) = \oplus_{i,j \in I} \ker(f)_{ij}$ is not in $\cc_J$. Then, there exist $s$, $t$ in $I$ but at least one of them is not in $J$ such that $\ker(f)_{st} \not= 0$. Thus, the composition of canonical inclusions $\ker(f)_{st} \to \ker(f) \to X$ is a monomorphism. This is a contradiction of Lemma \ref{morphism between component categories}. Hence, $\ker(f)$ is in $\cc_J$. Similarly, $\coker(f)$ and $\im(f)$ are also in $\cc_J$. Therefore, $\cc_J$ is an abelian category. Local finiteness and $\Bbbk$-linearity of $\cc_J$ are inherited from $\cc$ immediately. 

$\cc_J$ is a monoidal category where the tensor product is inherited from $\cc$ and the unit is $1_J$. The unit constraints are given by $l^J_X = i_J \otimes \id_X$ and $r^J_X = \id_X \otimes i_J$ for any object $X$ in $\cc$. The linearity and exactness of tensor product in $\cc_J$ are also inherited from $\cc$. Besides, for an object $X = \oplus_{i,j \in J} X_{ij}$ in $\cc_J$, one can check that $X^* = \oplus_{i,j \in J} X_{ij}^*$. By \cite[Remark 4.3.4 (4)]{Eti15}, each $X_{ij}^*$ is in $\cc_J$. Hence, $X^*$ is in $\cc_J$. 
\end{proof}

\begin{corollary}
Let $\cc$ be a multitensor category. For any non-empty subset $J$ of $I$, the subcategory $\cc_J$ of $\cc$ is also a multitensor category.    
\end{corollary}

\begin{proof}
Similar to Proposition \ref{CJ of multiring}, one can show that the right dual of objects in $\cc_J$ is also in $\cc_J$. 
\end{proof}

\begin{corollary}\label{CJ of multifusion}
Let $\cc$ be a multifusion category. For any non-empty subset $J$ of $I$, the subcategory $\cc_J$ of $\cc$ is also a multifusion category. 
\end{corollary}

\begin{proof}
Note that an object in $\cc_J$ is a simple object if and only if it is a simple object in $\cc$. Thus, $\cc_J$ has finitely many isomorphism classes of simple objects, since $\cc$ does. Besides, for any simple object $S$ in $\cc_J$, by Lemma \ref{morphism between component categories}, the projective cover $P(S)$ of $S$ in $\cc$ is in $\cc_J$. Indeed, $P(S)$ is also the projective cover of $S$ is $\cc_J$. Hence, $\cc_J$ has enough projectives.

For any object $X$ in $\cc_J$, we know that $X$ is a direct sum of some simple objects in $\cc$. Because $X$ is in $\cc_J$, these simple objects are also in $\cc_J$. Hence, $X$ is a direct sum of some simple objects in $\cc_J$. This means $\cc_J$ is semisimple. As a result, $\cc_J$ is a multifusion category. 
\end{proof}

Given monoidal categories $(\cc,\otimes,1)$ and $(\dd,\otimes,1')$, recall that  a functor $F:\cc\to\dd$ is said to be lax monoidal, see e.g. \cite[Definition 3.1]{AM10}, if there is a natural transformation $\phi^2$, given on components by $F(X)\otimes F(Y)\to F(X\otimes Y)$, and a morphism $\phi^0:1'\to F(1)$ such that
\begin{align*}
   \phi^2_{X, Y\otimes Z}(\id\otimes \phi^2_{Y,Z}) &=\phi^2_{X\otimes Y, Z}(\phi^2_{X,Y}\otimes \id),\\
l_{F(X)}=F(l_X)\phi^2_{1,X}(\phi^0\otimes \id),&\quad r_{F(X)}= F(r_X)\phi^2_{X,1}(\id\otimes \phi^0),
\end{align*}
where $l, r$ are the left and right unit constraints of the understood categories. Dually, one can define the colax monoidal functor, see e.g. \cite[Definition 3.2]{AM10}. Besides, a functor $F:\cc \to \dd$ is said to be a separable Frobenius monoidal functor if it is equipped with a lax monoidal structure $(\phi, \phi_0)$ and a colax monoidal structure $(\psi,\psi_0)$ such that, for every $X,Y,Z\in\cc$, the diagrams
\begin{align*}	
\xymatrixcolsep{1cm}\xymatrixrowsep{0.5cm}\xymatrix{
F(X\otimes Y)\otimes F(Z)\ar[r]^-{\psi_{X,Y}\otimes\id}\ar[d]_-{\phi_{X\otimes Y, Z}} & F(X)\otimes F(Y)\otimes F(Z)\ar[d]^-{\id\otimes \phi_{Y,Z}}\\
F(X\otimes Y\otimes Z)\ar[r]_-{\psi_{X,Y\otimes Z}}& F(X)\otimes F(Y\otimes Z)\\
F(X)\otimes F(Y\otimes Z)\ar[r]^-{\id\otimes\psi_{Y,Z}}\ar[d]_-{\phi_{X,Y\otimes Z}}&F(X)\otimes F(Y)\otimes F(Z)\ar[d]^-{\phi_{X,Y}\otimes\id}\\
F(X\otimes Y\otimes Z)\ar[r]_-{\psi_{X\otimes Y, Z}}&F(X\otimes Y)\otimes F(Z)
}
\end{align*}
are commutative, and $\phi_{X,Y}\psi_{X,Y}=\id_{F(X\otimes Y)}$. We refer readers to \cite[Definition 6.1]{Sz05} for the definition of separable Frobenius monoidal functor. In this paper, we use the notation from \cite[Definition 6.1]{BT15}.

Define the inclusion functor $L_J:\cc_J \to \cc$ by $L_J(X) = X$ and $L_J(f) = f$ for any object $X$ and morphism $f$ in $\cc_J$. We claim that the inclusion functor $L_J:\cc_J \to \cc$ for each non-empty subset $J$ of $I$ is a separable Frobenius monoidal functor.

\begin{proposition}\label{inclusion functor separable Frobenius}
Let $\cc$ be a multiring category with left duals. The inclusion functor $L_J:\cc_J \to \cc$ is a separable Frobenius monoidal functor.
\end{proposition}

\begin{proof}
For any object $X$, $Y$ in $\cc_J$, define $\phi_{X,Y}:L_J(X) \otimes L_J(Y) \to L_J(X\otimes Y)$ and $\psi_{X,Y}:L_J(X\otimes Y) \to L_J(X) \otimes L_J(Y)$ by identities. Define $\phi_0:= p_J: 1 \to 1_J$ and $\psi_0:= i_J: 1_J \to 1$. Since the left and right unit isomorphisms in $\cc_J$ are given by $l^J_X = l_X (i_J \otimes \id_X)$, $r^J_X = r_X(\id_X \otimes i_J)$ respectively, it suffices to show $l^J_X (p_J \otimes \id_X) = l_X$ and $r^J_X (\id_X \otimes p_J) = r_X$.   

Note that, for any object $X$ in $\cc_J$, there is a canonical isomorphism $\alpha: X\otimes (\oplus_{j\in I} 1_j) \to \oplus_{j\in I} (X \otimes 1_j)$. Since $X$ is in $\cc_J$, we know that $X \otimes 1_j = 0$ for any $j \notin J$. It follows that the canonical projection $p_{X}^J:\oplus_{j\in I} (X \otimes 1_j) \to \oplus_{j\in J} (X \otimes 1_j)$ is an isomorphism. Because $\id_X \otimes p_J = \beta p_{X}^J \alpha$, where $\beta:\oplus_{j\in J} (X \otimes 1_j) \to X \otimes 1_J$ is the canonical isomorphism, we know that $\id_X \otimes p_J$ is an isomorphism and its inverse is $\id_X \otimes i_J$. Therefore, $r^J_X (\id_X \otimes p_J) = r_X(\id_X \otimes i_J) (\id_X \otimes p_J) = r_X$. Similarly, one can show $l^J_X (p_J \otimes \id_X) = l_X$.
\end{proof}

Let $\cc$, $\dd$ be monoidal categories. It is well-known that a lax monoidal functor $(F:\cc\to\dd,\phi,\phi_0)$ maps an algebra $(A, m_A, u_A)$ to an algebra $(F(A), F(m_A)\circ \phi_{A,A}, F(u_A)\circ \phi_0)$, and a colax monoidal functor $(F:\cc\to\dd,\psi,\psi_0)$ maps a coalgebra $(C, \Delta_C, \varepsilon_C)$ to a coalgebra $(F(C), \psi_{C,C} \circ F(\Delta_C),\psi_0 \circ F(\varepsilon_C))$, see e.g. \cite[Proposition 3.29]{AM10}. Therefore, Proposition \ref{inclusion functor separable Frobenius} gives rise to the following result.

\begin{proposition}\label{algebra and coalgebra in component subcateogry}
Let $\cc$ be a multiring category with left duals. 

\begin{enumerate}
    \item Let $(A, m_A, u_A)$ be an algebra in $\cc_J$, where $u_A: 1_J \to A$. Then, $(A, m_A, u_Ap_J)$ is an algebra in $\cc$.
    \item Let $(C, \Delta_C, \varepsilon_C)$ be a coalgebra in $\cc_J$, where $\varepsilon_C: C \to 1_J$. Then, $(C, \Delta_C, i_J \varepsilon_C)$ is a coalgebra in $\cc$.   
\end{enumerate}
\end{proposition}

\begin{remark}
Suppose the unit object in $\cc$ is not simple. Let $A$ be an algebra in $\cc_J$ where $J$ is a proper subset of $I$. We observe that $-\otimes A: \cc \to \cc_A$ is not faithful. In fact, for any morphism $f$ in $\cc_{j,j}$ where $j$ in $I$ but not in $J$, we have $f\otimes A = 0$. This observation supports Proposition \ref{semisimple multiring main result}.
\end{remark}

For any object $X$ and any morphism $f$ in $\cc$, we denote by $X_J:= 1_J \otimes X \otimes 1_J$ and $f_J := 1_J \otimes f \otimes 1_J$. Define the projection functor $R_J: \cc \to \cc_J$ by $R_J(X) = X_J$ and $R_J(f) = f_J$. We claim that the projection functor $R_J$ is a lax and colax monoidal functor.

\begin{proposition}\label{projection functor lax}
Let $\cc$ be a multiring category with left duals. The projection functor $R_J: \cc \to \cc_J$ is a lax and colax monoidal functor.
\end{proposition}

\begin{proof}
For any object $X$, $Y$ in $\cc$, we define $\phi_{X,Y}:R_J(X) \otimes R_J(Y) \to R_J(X\otimes Y)$ by 
$$
\phi_{X,Y} := \id_{1_J \otimes X} \otimes i_J \otimes i_J \otimes \id_{Y\otimes 1_J},
$$
and $\psi_{X,Y}:R_J(X\otimes Y) \to R_J(X) \otimes R_J(Y)$ by 
$$
\psi_{X,Y}:= \id_{1_J \otimes X} \otimes p_J \otimes p_J \otimes \id_{Y\otimes 1_J}.
$$
Since $R_J(1) = 1_J \otimes 1 \otimes 1_J$, we also define $\phi_0: 1_J \to R(1)$ and $\psi_0:R(1) \to 1_J$ by $\phi_0 = \id_{1_J}\otimes \id_1 \otimes p_J$ and $\psi_0= \id_{1_J}\otimes \id_1 \otimes i_J$ i.e. the canonical isomorphisms. Therefore,
\begin{align*}
&\phi_{X\otimes Y, Z} (\phi_{X,Y} \otimes \id_{R_J(Z)})\\ 
= &(\id_{1_J \otimes X \otimes Y} \otimes i_J \otimes i_J \otimes \id_{Z\otimes 1_J})(\id_{1_J \otimes X} \otimes i_J \otimes i_J \otimes \id_{Y\otimes 1_J} \otimes \id_{R_J(Z)})\\
= &\id_{1_J \otimes X} \otimes i_J \otimes i_J \otimes \id_Y \otimes i_J \otimes i_J \otimes \id_{Z\otimes 1_J}\\
= &(\id_{1_J \otimes X} \otimes i_J \otimes i_J \otimes \id_{Y \otimes Z\otimes 1_J})(\id_{1_J \otimes X \otimes 1_J} \otimes \id_{1_J \otimes Y} \otimes i_J \otimes i_J \otimes \id_{Z\otimes 1_J})\\
= &\phi_{X,Y\otimes Z} (\id_{R_J(X)} \otimes \phi_{Y,Z}),
\end{align*}
and similarly one can obtain $ (\psi_{X,Y} \otimes \id_{R_J(Z)})\psi_{X\otimes Y, Z}=  (\id_{R_J(X)} \otimes \psi_{Y,Z})\psi_{X,Y\otimes Z}$. It remains to show the unitality. Because in any monoidal category $l_{1} = r_{1}$, see e.g. \cite[Corollary 2.2.5]{Eti15}, we apply it to $\cc_J$ and obtain $i_J \otimes \id_{1_J} = l^J_{1_J} = r^J_{1_J} = \id_{1_J} \otimes i_J$, where $l^J$ and $r^J$ are the unit constraints in $\cc_J$. Besides, similar to the proof of Proposition \ref{inclusion functor separable Frobenius}, we have $\id_{1_J} \otimes i_J$ is an isomorphism with inverse $\id_{1_J} \otimes p_J$. Thus,
\begin{align*}
R_J(l_X) \phi_{1,X} (\phi_0 \otimes \id_{R_J(X)}) = &(\id_{1_J} \otimes i_J \otimes i_J \otimes \id_{X\otimes 1_J})(\id_{1_J} \otimes p_J \otimes \id_{R_J(X)})\\
= &\id_{1_J} \otimes i_J  \otimes \id_{X\otimes 1_J} = i_J \otimes \id_{1_J} \otimes \id_{X\otimes 1_J} = l^J_{R_J(X)}.
\end{align*}
Moreover, since $i_J\otimes \id_{1_J} = \id_{1_J} \otimes i_J$, we have
\begin{align*}
&R_J(r_X) \phi_{X,1} (\id_{R_J(X)} \otimes \phi_0) = (\id_{1_J \otimes X} \otimes i_J \otimes i_J \otimes \id_{1_J}) (\id_{R_J(X)} \otimes \id_{1_J} \otimes p_J)\\
= &(\id_{1_J \otimes X} \otimes i_J \otimes \id_1 \otimes \id_{1 \otimes 1_J})(\id_{1_J \otimes X} \otimes \id_{1_J} \otimes i_J \otimes \id_{1_J}) (\id_{R_J(X)} \otimes \id_{1_J} \otimes p_J)\\
= &(\id_{1_J \otimes X} \otimes i_J \otimes \id_{1_J})(\id_{1_J \otimes X} \otimes \id_{1_J} \otimes \id_{1_J} \otimes i_J) (\id_{R_J(X)} \otimes \id_{1_J} \otimes p_J)\\
= &\id_{1_J \otimes X} \otimes i_J \otimes \id_{1_J} = \id_{1_J \otimes X} \otimes \id_{1_J} \otimes i_J = r^J_{R_J(X)}.
\end{align*}

Similarly, one can show $l^J_{R_J(X)}(\psi_0 \otimes \id_{R_J(X)}) \psi_{1,X}= R_J(l_X)$, and $r^J_{R_J(X)}(\id_{R_J(X)} \otimes \psi_0) \psi_{X,1}= R_J(r_X)$. Thus, $R_J: \cc \to \cc_J$ is a lax and colax monoidal functor.
\end{proof}

\begin{remark}
Recall that a Frobenius pair $(F, G)$ is a pair of functors $F:\cc\to \dd$ and $G:\dd \to \cc$ such that $G$ is at the same time a left and a right adjoint to $F$. We claim that $(L_J, R_J)$ is a Frobenius pair. In fact, for any object $B$ in $\cc_J$ and $A$ in $\cc$, by Lemma \ref{morphism between component categories}, we have $\Hom_{\cc}(B, A_{jk}) = 0 = \Hom_{\cc}(A_{jk}, B)$ unless $j,k$ are in $J$. Since $I$ is a finite set, we obtain
$$
\begin{aligned}
\Hom_{\cc}(L_J(B), A&) = \Hom_{\cc}(B, A) \cong \Hom_{\cc}(B, \oplus_{j,k  \in I} A_{jk}) \cong \oplus_{j,k \in I} \Hom_{\cc}(B, A_{jk})\\
= &\oplus_{j,k \in J} \Hom_{\cc}(B, A_{jk}) \cong \Hom_{\cc}(B, \oplus_{j,k \in J} A_{jk}) \cong \Hom_{\cc_J}(B, R_J(A)),
\end{aligned}
$$
and
$$
\begin{aligned}
\Hom_{\cc}(A, L_J(B)&) = \Hom_{\cc}(A, B) \cong \Hom_{\cc}(\oplus_{j,k  \in I} A_{jk}, B) \cong \oplus_{j,k \in I} \Hom_{\cc}(A_{jk}, B)\\
= &\oplus_{j,k \in J} \Hom_{\cc}(A_{jk}, B) \cong \Hom_{\cc}(\oplus_{j,k \in J} A_{jk}, B) \cong \Hom_{\cc_J}(R_J(A), B).
\end{aligned}
$$
In fact, the two isomorphisms above are both obtained by applying the functor $R_J$ together with the identity $R_JL_J= \id$, so they are natural on $A$ and $B$. Therefore, $(L_J, R_J)$ is a Frobenius pair. This means that $\cc_J$ is both a reflective and a coreflective subcategory in the sense of \cite[Definition 3.5.2]{Bor94}.
\end{remark}

By means of the projection functor, we obtain the algebra structure of $A_J$ in $\cc_J$.

\begin{proposition}\label{component algebra}
Let $\cc$ be a multiring category with left duals. 
\begin{enumerate}
    \item Let $(A, m_A, u_A)$ be an algebra in $\cc$. Then, $A_J$ is an algebra in $\cc_J$ for which the multiplication $m_{A_J}$ and the unit $u_{A_J}: 1_J \to A_J$ satisfy that $i_{A_J} m_{A_J}=m_A(i_{A_J}\otimes i_{A_J})$ and $i_{A_J}u_{A_J} p_J= u_A$, where $i_{A_J}:= i_J \otimes \id_A \otimes i_J: 1_J \otimes A \otimes 1_J \to A$ is the canonical inclusion. 
    \item Let $(C, \Delta_C, \varepsilon_C)$ be a coalgebra in $\cc$. Then, $C_J$ is a coalgebra in $\cc_J$ for which the comultiplication $\Delta_{C_J}$ and the counit $\varepsilon_{C_J}: C_J \to 1_J$ satisfy that $\Delta_{C_J} p_{C_J}=(p_{C_J} \otimes p_{C_J})\Delta_C$ and $i_J \varepsilon_{C_J} p_{C_J} = \varepsilon_C$, where $p_{C_J}:= p_J \otimes \id_C \otimes p_J: C \to 1_J \otimes C \otimes 1_J$ is the canonical projection.   
\end{enumerate}

\end{proposition}

\begin{proof}
$(1)$ By Proposition \ref{projection functor lax} and \cite[Proposition 3.29]{AM10}, we know $A_J = R_J(A)$ is an algebra in $\cc_J$ with multiplication $R_J(m_A) \phi_{A,A}$ and unit $R_J(u_A) \phi_0$. Consequently, we have
\begin{align*}
i_{A_J} R_J(m_A) \phi_{A,A} = &(i_J \otimes m_A \otimes i_J)(\id_{1_J \otimes A} \otimes i_J \otimes i_J \otimes \id_{A\otimes 1_J})\\
= &m_A (i_J \otimes \id_{A} \otimes i_J \otimes i_J \otimes \id_{A}\otimes i_J) = m_A(i_{A_J}\otimes i_{A_J}),
\end{align*} 
and
$$
i_{A_J}u_{A_J} p_J = i_{A_J} R_J(u_A) \phi_0 p_J= (i_J \otimes u_A \otimes i_J) (p_J \otimes p_J) = u_A.
$$

$(2)$ By Proposition \ref{projection functor lax} and \cite[Proposition 3.29]{AM10}, we know $C_J$ is a coalgebra in $\cc_J$ with multiplication $\psi_{C,C}R_J(\Delta_C)$ and unit $\psi_0 R_J(\varepsilon_C)$. It follows that
\begin{align*}
\psi_{C,C}R_J(\Delta_C) p_{C_J}
= &(\id_{1_J \otimes C} \otimes p_J \otimes p_J \otimes \id_{C \otimes 1_J}) (\id_{1_J} \otimes \Delta_C \otimes \id_{1_J}) (p_J \otimes C \otimes p_J)\\
= &(p_J \otimes \id_C \otimes p_J \otimes p_J \otimes \id_C \otimes p_J) \Delta_C = (p_{C_J} \otimes p_{C_J})\Delta_C,   
\end{align*}
and
\begin{align*}
i_J \varepsilon_{C_J} p_{C_J} = &i_J \psi_0 R(\varepsilon_C) p_{C_J} = (i_J \otimes \varepsilon_C \otimes i_J)(p_J \otimes C \otimes p_J) = \varepsilon_C.
\end{align*}
\end{proof}

\begin{invisible}
\begin{proposition}
Let $\cc$ be a multiring category with left duals, and $(A, m_A, u_A)$ be an algebra in $\cc$. Then, $A_{ii}$ is an algebra in $\cc_{ii}$ for which the multiplication $m_{A_{ii}}$ and the unit $u_{A_{ii}}: 1_i \to A_{ii}$ are the unique morphisms induced by the universal property of kernel such that $i_{A_{ii}} m_{A_{ii}}=m_A(i_{A_{ii}}\otimes i_{A_{ii}})$ and $i_{A_{ii}}u_{A_{ii}}= u_A i_i$.  
\end{proposition}

\begin{proof}
We denote the cokernel of $i_{A_{ii}}$ by $\pi: A \to A/A_{ii}$. Obseve that $\pi m_A(i_{A_{ii}}\otimes i_{A_{ii}}): A_{ii} \otimes A_{ii} \to A/A_{ii}$ and $\pi u_A i_i: 1_i \to A/A_{ii}$ are morphisms from $\cc_{ii}$ to $\oplus_{(s,t)\not= (i,i)} \cc_{st}$. By Lemma \ref{morphism between component categories}, $\pi m_A(i_{A_{ii}}\otimes i_{A_{ii}})$ and $\pi u_A i_i$ are zero. Note that $i_{A_{ii}}$ is the kernel of $\pi$. by the universal property of kernel, we obtain $m_{A_{ii}}$ and $u_{A_{ii}}: 1_i \to A_{ii}$ such that $i_{A_{ii}} m_{A_{ii}}=m_A(i_{A_{ii}}\otimes i_{A_{ii}})$ and $i_{A_{ii}}u_{A_{ii}}= u_A i_i$. Then, we compute the algebra structure. Because
\begin{align*}
&i_{A_{ii}} m_{A_{ii}} (m_{A_{ii}} \otimes \id_{A_{ii}}) = m_A(i_{A_{ii}}m_{A_{ii}} \otimes i_{A_{ii}}) = m_A( m_A(i_{A_{ii}}\otimes i_{A_{ii}}) \otimes i_{A_{ii}})\\
= &m_A(m_A \otimes \id_A) (i_{A_{ii}}\otimes i_{A_{ii}} \otimes i_{A_{ii}}) = m_A(\id_A \otimes m_A) (i_{A_{ii}}\otimes i_{A_{ii}} \otimes i_{A_{ii}})\\
= &m_A (i_{A_{ii}}\otimes m_A(i_{A_{ii}} \otimes i_{A_{ii}})) = m_A (i_{A_{ii}}\otimes i_{A_{ii}} m_{A_{ii}}) = i_{A_{ii}} m_{A_{ii}} (\id_{A_{ii}} \otimes m_{A_{ii}}), 
\end{align*}
we obtain $m_{A_{ii}} (m_{A_{ii}} \otimes \id_{A_{ii}}) = m_{A_{ii}} (\id_{A_{ii}} \otimes m_{A_{ii}})$. Furthermore, since
\begin{align*}
&i_{A_{ii}} m_{A_{ii}} (u_{A_{ii}} \otimes \id_{A_{ii}}) = m_A(i_{A_{ii}}u_{A_{ii}} \otimes i_{A_{ii}}) = m_A(u_A i_i \otimes i_{A_{ii}})\\ 
= &m_A(u_A \otimes \id_A) (i_i \otimes i_{A_{ii}}) = l_A (i_i \otimes i_{A_{ii}}) = l_A (\id_{1} \otimes i_{A_{ii}}) (i_i \otimes \id_{A_{ii}})\\
= &i_{A_{ii}} l_{A_{ii}}(i_i \otimes \id_{A_{ii}}) = i_{A_{ii}} l'_{A_{ii}},
\end{align*}
where $l'_{A_{ii}} = l_{A_{ii}}(i_i \otimes \id_{A_{ii}})$ is the left unit isomorphism of $A_{ii}$ in $\cc_{ii}$, we have $m_{A_{ii}} (u_{A_{ii}} \otimes \id_{A_{ii}}) = l'_{A_{ii}}$. Similarly, one can show $m_{A_{ii}} (\id_{A_{ii}}\otimes u_{A_{ii}}) = r'_{A_{ii}}$. As a result, $A_{ii}$ is an algebra in $\cc_{ii}$.
\end{proof}
\end{invisible}

\begin{corollary}\label{AJ is algebra in C}
Let $\cc$ be a multiring category with left duals. 
\begin{enumerate}
    \item Let $(A, m_A, u_A)$ be an algebra in $\cc$. Then, $A_J$ is an algebra in $\cc$, and $i_{A_J}:A_J \to A$ is an algebra morphism. 
    \item Let $(C, \Delta_C, \varepsilon_C)$ be a coalgebra in $\cc$. Then, $C_J$ is a coalgebra in $\cc$, and $p_{C_J}:C \to C_J$ is a coalgebra morphism.
\end{enumerate}    
\end{corollary}

\begin{proof}
$(1)$ By Proposition \ref{component algebra}, we know $A_J$ is an algebra in $\cc_J$ satisfying $i_{A_J} m_{A_J}=m_A(i_{A_J}\otimes i_{A_J})$ and $i_{A_J}u_{A_J} p_J= u_A$. By Proposition \ref{algebra and coalgebra in component subcateogry}, $(A_J, m_{A_J}, u_{A_J} p_J)$ is an algebra in $\cc$. The two equations imply that $i_{A_J}:A_J \to A$ is an algebra morphism. $(2)$ It follows similarly.
\end{proof}

The following proposition proves that any non-zero algebra $A = \oplus_{i,j \in I} A_{ij}$ in $\cc$ must be of the form $A_J$ where $J=\{ i\in I| A_{ii} \not= 0 \}$. 

\begin{proposition}\label{equal zero}
Let $\cc$ be a multiring category with left duals, and $(A, m_A, u_A)$ be an algebra in $\cc$. Suppose $A_{ll} = 0$ for $l$ in $I$. Then, $A_{lj} = A_{jl}= 0$ for all $j$ in $I$. Consequently, $A = A_J$ for $J=\{ i\in I| A_{ii} \not= 0 \}$.   
\end{proposition}

\begin{proof}
Because $1 = \oplus_{i\in I} 1_i$, we know that $1$ is in $\oplus_{i\in I} \cc_{ii}$. Thus, by Lemma \ref{morphism between component categories}, $\im(u_A)$ is in $\oplus_{i\in I} \cc_{ii}$. We can write $\im(u_A) = \oplus_{i\in I} \im(u_A)_{ii}$ where $\im(u_A)_{ii}$ is in $\cc_{ii}$ for each $i \in I$. Consequently, 
$$
\im(u_A) \otimes A_{lj} = (\oplus_{i\in I} \im(u_A)_{ii}) \otimes A_{lj} \cong \oplus_{i\in I} (\im(u_A)_{ii} \otimes A_{lj}) = \im(u_A)_{ll} \otimes A_{lj},
$$
and similarly $A_{jl} \otimes \im(u_A) \cong A_{jl} \otimes \im(u_A)_{ll}$. Note that $\im(u_A)$ is a subobject of $A$. Hence, $\im(u_A)_{ii}$ is a subobject of $A_{ii}$, which means $\im(u_A)_{ll} = 0$ as $A_{ll} = 0$. Therefore, $\im(u_A) \otimes A_{lj} \cong \im(u_A)_{ll} \otimes A_{lj} = 0$ and $A_{jl} \otimes \im(u_A) \cong A_{jl} \otimes \im(u_A)_{ll} = 0$. As a consequence, $u_A \otimes \id_{A_{lj}} = \id_{A_{jl}} \otimes u_A = 0$. It follows that
$$
i_{A_{lj}} = m_A (u_A \otimes \id_A) (\id_1 \otimes i_{A_{lj}}) = m_A (\id_A \otimes i_{A_{lj}}) (u_A \otimes \id_{A_{lj}}) = 0
$$
and
$$
i_{A_{jl}} = m_A (\id_A \otimes u_A) (i_{A_{jl}} \otimes \id_1) = m_A (i_{A_{jl}} \otimes \id_A) (\id_{A_{jl}} \otimes u_A) = 0.
$$
As a result, $A_{lj} = A_{jl}= 0$.
\end{proof}

\begin{remark}
Let $\cc$ be a multiring category with left duals. For any algebra $A$ in $\cc$ such that $J = \{ i \in I | A_{ii} \not= 0\}$ contains at least two elements, one can observe that $u_{A_{jj}}: 1 \to A_{jj}$ for any $j \in J$ is not a monomorphism. In fact, for any $k \not= j$, $u_{A_{jj}} i_k$ is a morphism from $\cc_{kk}$ to $\cc_{jj}$, which results to be zero by Lemma \ref{morphism between component categories}. This observation also supports Proposition \ref{semisimple multiring main result}. Besides, one can easily observe that $u_{A}$ is a monomorphism if and only if $J=I$.
\end{remark}

\begin{proposition}\label{unit mono}
Let $\cc$ be a multiring category with left duals, and $(A, m_A, u_A)$ be an algebra in $\cc$. Let $J=\{ i\in I| A_{ii} \not= 0 \}$. Then, $u_{A_J}: 1_J \to A_J$ is a monomorphism.  
\end{proposition}

\begin{proof}
Note that $\ker(u_{A_J})$ is in $\oplus_{i\in J} \cc_{ii}$ since it is a subobject of $1_J$. Thus, we can write
\begin{align*}
\ker(u_{A_J}) = &\oplus_{i \in J} \ker(u_{A_J})_{ii} = \oplus_{i \in J} (1_i \otimes \ker(u_{A_J}) \otimes 1_i) 
\cong \oplus_{i \in J} \ker(1_i \otimes u_{A_J} \otimes 1_i)\\
= &\oplus_{i \in J} \ker(1_i \otimes 1_J \otimes u_{A} \otimes 1_J \otimes 1_i) \cong \oplus_{i \in J} \ker(1_i \otimes u_{A} \otimes 1_i).
\end{align*}
For any $i$ in $J$, because $\ker(1_i \otimes u_A \otimes 1_i)$ is a subobject of $1_i$, and $1_i$ is a simple object, we know that $\ker(1_i \otimes u_A \otimes 1_i)$ is isomorphic to either $1_i$ or $0$. Since $A_{ii} \not= 0$, by Proposition \ref{component algebra} we know that $A_{ii}$ is a non-zero algebra in $\cc_{ii}$, which implies $u_{A_{ii}} \not= 0$. Hence, $1_i \otimes u_A \otimes 1_i$ is non-zero. This means $\ker(1_i \otimes u_A \otimes 1_i) = 0$. Consequently, $\ker(u_{A_J}) = 0$.
\end{proof}

\begin{proposition}\label{restriction}
Let $\cc$ be a semisimple multiring category with left duals, and $(A, m_A, u_A)$ be an algebra in $\cc$. Let $J=\{ i\in I| A_{ii} \not= 0 \}$. Then, the functor $-\otimes A_J: \cc_J \to (\cc_J)_{A_J}$ is separable.
\end{proposition}

\begin{proof}
By Proposition \ref{semisimple regular}, we know $u_{A_J}: 1_J \to A_J$ is regular.
Since $u_{A_J}$ is a monomorphism, see Proposition \ref{unit mono}, we obtain $u_{A_J}$ is a split-mono. Therefore, by Proposition \ref{induction functor}, we obtain $-\otimes A_J: \cc_J \to (\cc_J)_{A_J}$ is separable.
\end{proof}

\begin{remark}
Recall that for monoidal categories $\cc$ and $\dd$ and an algebra $A$ in $\cc$, if $F:\cc \to \dd$ is a lax monoidal functor with $\phi_0$ split-epi, then $-\otimes A: \cc \to \cc_A$ is semiseparable implies that $-\otimes F(A): \dd \to \dd_{F(A)}$ is semiseparable, see \cite[Proposition 3.25]{BZ26}. By Proposition \ref{semisimple semiseparable}, $-\otimes A: \cc \to \cc_A$ is semiseparable. Thus, Proposition \ref{restriction} indicates that $-\otimes F(A): \dd \to \dd_{F(A)}$ could be separable, where $F = R_J$ is a lax monoidal functor by Proposition \ref{projection functor lax}. Furthermore, recall that if $F:\cc \to \dd$ is a fully faithful lax monoidal functor with $\phi_0$ epimorphism, then $-\otimes F(A): \dd \to \dd_{F(A)}$ is separable implies that $-\otimes A: \cc \to \cc_A$ is separable, see \cite[Proposition 3.26]{BZ26}. In Proposition \ref{restriction}, the functor $F = R_J$ is not faithful. This implies the faithfulness in \cite[Proposition 3.26]{BZ26} is indispensable.
\end{remark}

\section{Applications and examples}\label{Applications and examples}

This section is divided into two parts: Applications and Examples. We mainly focus on the applications and examples of Proposition \ref{main result}. For the sake of brevity, we only mention the equivalent condition $(2)$ in Proposition \ref{main result}. In fact, one can also use the other equivalent conditions to obtain similar results. 

\subsection{Application to weak Hopf algebras}\label{Application to weak Hopf algebras}

Now, we consider an application of Proposition \ref{main result} to a certain class of weak Hopf algebras. Recall \cite{Ha99}, \cite{Sz01}, see also \cite[Corollary 2.22]{ENO05} that any multifusion category $\cc$ is equivalent to the category of finite dimensional representations of a (nonunique) regular semisimple weak Hopf algebra. This weak Hopf algebra is connected if and only if $\cc$ is a fusion category. Indeed, this weak Hopf algebra can be given by $H=\mathrm{End}(F)$, see e.g. \cite[2.5]{ENO05}, where $F:\cc \to R\!-\!\mathrm{Bimon}$ is a fiber functor. By Proposition \ref{main result}, we know that $\cc$ is a fusion category if and only if for any non-zero $H$-module algebra $A$, the functor $-\otimes A: \mathrm{Rep}(H) \to \mathrm{Rep}(H)_A$ is separable. Therefore, we establish the correspondence between connectness of $H$ and the separability.

\begin{proposition}
$H$ is connected if and only if for any non-zero $H$-module algebra $A$, the functor $-\otimes A: \mathrm{Rep}(H) \to \mathrm{Rep}(H)_A$ is separable.
\end{proposition}

We can also apply Proposition \ref{restriction} to $\mathrm{Rep}(H)$. Suppose $\cc$, hence $\mathrm{Rep}(H)$, is a multifusion category but not a fusion category. Consider $\mathrm{Rep}(H) :=\oplus_{i,j \in I} \mathrm{Rep}(H)_{ij}$. For a non-zero $H$-module algebra $A$, let $J=\{ i\in I| A_{ii} \not= 0 \}$. By Proposition \ref{restriction}, $-\otimes A_J: \mathrm{Rep}(H)_J \to (\mathrm{Rep}(H)_J)_{A_J}$ is separable, where $\mathrm{Rep}(H)_J :=\oplus_{i,j \in J} \mathrm{Rep}(H)_{ij}$ is the direct sum of component subcategories of $\mathrm{Rep}(H)$ indexed by $i,j \in J$. By Corollary \ref{CJ of multifusion}, we know that $\mathrm{Rep}(H)_J$ is a multifusion category. Therefore, $\mathrm{Rep}(H)_J$ is equivalent to the category of finite dimensional representations of a (nonunique) regular semisimple weak Hopf algebra. This weak Hopf algebra can be given by $H'=\mathrm{End}(F')$, where $F':\mathrm{Rep}(H)_J \to R'\!-\!\mathrm{Bimon}$ is a fiber functor. Consequently, we can replace $\mathrm{Rep}(H)_J$ by $\mathrm{Rep}(H')$.

\begin{proposition}
For a non-zero $H$-module algebra $A$, let $J=\{ i\in I| A_{ii} \not= 0 \}$. Denote the equivalence by $E:\mathrm{Rep}(H)_J \to \mathrm{Rep}(H')$. Then, the functor $-\otimes E(A_J): \mathrm{Rep}(H') \to (\mathrm{Rep}(H'))_{E(A_J)}$ is separable.
\end{proposition}

\begin{remark}
Let $\cc$ be a multiring category with left duals. For any equivalence $G:\cc \to \dd$ with its quasi-inverse $F:\dd \to \cc$, we know that there are natural isomorphisms $\eta: \id_{\cc} \to FG$ and $\eta': \id_{\dd} \to GF$. For the sake of brevity, we abuse the notation $1$ for the unit in both $\cc$ and $\dd$. Define natural transformations 
\begin{align*}
&1_J \otimes \eta \otimes 1_J: 1_J \otimes \id_{\cc} \otimes 1_J \to 1_J \otimes FG \otimes 1_J\\
&1_J \otimes \eta' \otimes 1_J: 1_J \otimes \id_{\dd} \otimes 1_J \to 1_J \otimes GF \otimes 1_J
\end{align*}
by
\begin{align*}
&(1_J \otimes \eta \otimes 1_J)_X := 1_J \otimes \eta_X \otimes 1_J: 1_J \otimes X \otimes 1_J \to 1_J \otimes FG(X) \otimes 1_J\\
&(1_J \otimes \eta' \otimes 1_J)_Y := 1_J \otimes \eta'_Y \otimes 1_J: 1_J \otimes Y \otimes 1_J \to 1_J \otimes GF(Y) \otimes 1_J\\   
\end{align*}
for $X$ in $\cc$ and $Y$ in $\dd$. One can readily observe that $1_J \otimes \eta \otimes 1_J$ and $1_J \otimes \eta' \otimes 1_J$ are natural isomorphisms. Therefore, $1_J \otimes G \otimes 1_J: 1_J \otimes \cc \otimes 1_J \to 1_J \otimes \dd \otimes 1_J$ is an equivalence. This means that $\cc_J$ is equivalent to $\dd_J$. In this subsection, we can replace $\dd$ by $\mathrm{Rep}(H)$, and obtain that $\cc_J$ is equivalent to $\mathrm{Rep}(H)_J$.
\end{remark}

\subsection{Application to transfer of simplicity}\label{Application to transfer of simplicity}

In this subsection, we consider the transfer of the simplicity of the unit object by a monoidal functor. First, we recall the following result.

\begin{proposition}\cite[Proposition 3.26]{BZ26}\label{retraction of semiseparability}
Let $(\cc, \otimes, 1)$, $(\dd, \otimes, 1')$ be monoidal categories, $A$ be an algebra in $\cc$ with unit $u_A:1\to A$. Let $(F:\cc\to\dd,\phi^2,\phi^0)$ be a fully faithful lax monoidal functor. If $\phi^0$ is an epimorphism and $-\otimes F(A):\dd \to \dd_{F(A)}$ is semiseparable (resp., separable, naturally full), then $-\otimes A:\cc\to\cc_A$ is semiseparable (resp., separable, naturally full).
\end{proposition}

Recall again that a lax monoidal functor $(F:\cc\to\dd,\phi^2,\phi^0)$ maps an algebra $(R, m_R, u_R)$ in $\cc$ into an algebra $(F(R), F(m_R) \phi^2_{R,R}, F(u_R) \phi^0)$ in $\dd$, see e.g.\ \cite[Proposition 3.29]{AM10}.

\begin{proposition}
Let $(\cc, \otimes, 1)$ be a semisimple multiring category with left duals, and $(\dd, \otimes, 1')$ be a semisimple abelian monoidal category with biexact tensor product such that the unit object is simple. Suppose there is a fully faithful lax monoidal functor $(F:\cc\to\dd,\phi^2,\phi^0)$ such that $\phi^0$ is an epimorphism. Then, $\cc$ is a ring category.    
\end{proposition}

\begin{proof}
For any non-zero algebra $A$ in $\cc$, we know that $F(A)$ is an algebra in $\dd$. By Proposition \ref{semisimple induction separable}, we obtain that $-\otimes F(A):\dd \to \dd_{F(A)}$ is separable. By Proposition \ref{retraction of semiseparability}, $-\otimes A:\cc\to\cc_A$ is separable. Consequently, by Proposition \ref{semisimple multiring main result}, we obtain that $\cc$ is a ring category.
\end{proof}

One can apply this result to semisimple multiring categories with left duals and multifusion categories, then obtain the following corollary immediately.

\begin{corollary}
Let $(\cc, \otimes, 1)$, $(\dd, \otimes, 1')$ be semisimple multiring categories with left duals (resp., multifusion categories). Suppose $\dd$ is a ring category and there is a fully faithful lax monoidal functor $(F:\cc\to\dd,\phi^2,\phi^0)$ such that $\phi^0$ is an epimorphism. Then, $\cc$ is also a ring category (resp., fusion category).    
\end{corollary}

\subsection{Application to semisimple indecomposable module category}\label{Application to semisimple indecomposable module category}

The theory of module categories over tensor categories is crucial for understanding the structure of tensor categories. A systematic theory is developed in \cite[Chapter 7]{Eti15}. In this subsection, we concentrate on semisimple indecomposable module categories. 

The definition of indecomposable module category can be found in \cite[Definition 2.6, 2.7]{Ost03}. More precisely, a module category over a monoidal category $\cc$ is a category $\m$ together with an exact bifunctor $\otimes:\cc \times \m \to \m$ and functorial associativity and unit isomorphisms $m_{X,Y,M}:(X\otimes Y) \otimes M \to X \otimes (Y \otimes M)$, $l_M: 1\otimes M \to M$ for any $X,Y \in \cc$, $M \in \m$ such that $(\id_X \otimes m_{Y,Z,M})m_{X, Y \otimes Z, M}(a_{X,Y,Z}\otimes \id_M) = m_{X,Y,Z\otimes M} m_{X\otimes Y,Z,M}$ and $(\id_X \otimes l_M) m_{X,1,Y} = r_X \otimes \id_M$. A module category is indecomposable if it is not equivalent to a direct sum of two nontrivial module categories. Next, we apply Proposition \ref{main result} to this setting. 

Let $\cc$ be a fusion category, and $\mathcal{M}$ be a semisimple indecomposable module category over $\cc$. By \cite[Theorem 3.1]{Ost03}, $\mathcal{M}$ is equivalent to $\cc_A$, where $A:= \underline{\Hom}(M,M)$ for a fixed object $M$ in $\mathcal{M}$. Recall also that an equivalence is separable, see e.g. \cite[Corollary 9]{CMZ02}. By Proposition \ref{main result}, the functor $-\otimes A:\cc \to \cc_A$ is separable. Thus, their composition $\cc \to \cc_A \to \mathcal{M}$ is still a separable functor, by \cite[Proposition 46]{CMZ02}. Hence, we have the following proposition.

\begin{proposition}
Let $\cc$ be a fusion category. For any semisimple indecomposable module category $\mathcal{M}$ over $\cc$, the above-mentioned composition $\cc \to \cc_A \to \mathcal{M}$ is a separable functor from $\cc$ to $\mathcal{M}$.    
\end{proposition}

\subsection{Application to Grothendieck ring}\label{Application to Grothendieck ring}

First, we recall the definition of $\mathbb{Z}_+$-ring, see e.g. \cite[Definition 3.1.1]{Eti15}. Let $\mathbb{Z}_+$ denote the semi-ring of non-negative integers. Let $A$ be a ring which is free as a $\mathbb{Z}$-module. A $\mathbb{Z}_+$-basis of $A$ is a basis $B = \{b_i\}_{i\in I}$ such that $b_ib_j = \sum_{k\in I} c^k_{ij} b_k$, where $c^k_{ij} \in \mathbb{Z}_+$. A $\mathbb{Z}_+$-ring is a ring with a fixed $\mathbb{Z}_+$-basis and with unit $1$ which is a non-negative linear combination of the basis elements. A unital $\mathbb{Z}_+$-ring is a $\mathbb{Z}_+$-ring such that $1$ is a basis element. 

Let A be a $\mathbb{Z}_+$-ring, and let $I_0$ be the set of $i \in I$ such that $b_i$ occurs in the decomposition of $1$. Let $\tau:A \to \mathbb{Z}$ denote the group homomorphism defined by $\tau(b_i) = 1$ if $i \in I_0$, $\tau(b_i) = 0$ if $i \not= j^*$. A $\mathbb{Z}_+$-ring $A$ with basis $\{b_i\}_{i\in I}$ is called a based ring if there exists an involution $i \mapsto i^*$ of I such that the induced map $a = \sum_{i\in I} a_i b_i \mapsto a^* = \sum_{i\in I} a_i b_{i^*}$ is an anti-involution of the ring $A$, and $\tau(b_i b_j) = 1$ if $i = j^*$, $\tau(b_i b_j) = 0$ if $i \not= j^*$, see e.g. \cite[Definition 3.1.3]{Eti15}. Recall that for a given $\mathbb{Z}_+$-ring A, being a (unital) based ring is a property, not an additional structure, see e.g. \cite[Exercise 3.1.5 (ii)]{Eti15}.

Let $\cc$ be an abelian $\Bbbk$-linear category where objects have finite length. Recall that the Grothendieck group $\mathrm{Gr}(\cc)$ of $\cc$ is the free abelian group generated by isomorphism classes of simple objects in $\cc$, see e.g. \cite[Definition 1.5.8]{Eti15}. Now, let $\cc$ be a multiring category. The tensor product on $\cc$ induces a
natural multiplication on $\mathrm{Gr}(\cc)$. Thus, $\mathrm{Gr}(\cc)$ is a $\mathbb{Z}_+$-ring with the unit $[1]$, which is called the Grothendieck ring, see e.g. \cite[Definition 4.5.2]{Eti15}. Besides, a multifusion ring is a based ring of finite rank, and a fusion ring is a unital based ring of finite rank, see e.g. \cite[Definition 3.1.7]{Eti15}. Now, we recall a famous result which connect (multi)fusion category and (multi)fusion ring.

\begin{proposition}\cite[Proposition 4.9.1]{Eti15}\label{fusion category fusion ring}
If $\cc$ is a semisimple multitensor category then $\mathrm{Gr}(\cc)$ is a
based ring. If $\cc$ is a (multi)fusion category, then $\mathrm{Gr}(\cc)$ is a (multi)fusion ring.
\end{proposition}

Based on this proposition, our result induces a new connection between fusion ring and separability.

\begin{proposition}
Let $\cc$ be a multifusion category. Then, $\mathrm{Gr}(\cc)$ is a fusion ring if and only if the functor $-\otimes A:\cc \to \cc_A$ is separable for any non-zero algebra $A$ in $\cc$. 
\end{proposition}

\begin{proof}
$(\Rightarrow)$ Suppose $\mathrm{Gr}(\cc)$ is a fusion ring. By definition of fusion ring, we know that the isomorphism class $[1]$ of unit object $1$ in $\cc$ is a basis element. This implies that the unit object $1$ in $\cc$ is a simple object. Consequently, $\cc$ is a fusion category. By Proposition \ref{main result}, the functor $-\otimes A:\cc \to \cc_A$ is separable for any non-zero algebra $A$ in $\cc$.

$(\Leftarrow)$ Suppose for any non-zero algebra $A$ in $\cc$ the functor $-\otimes A:\cc \to \cc_A$ is separable. By Proposition \ref{main result}, $\cc$ is a fusion category. Then, by Proposition \ref{fusion category fusion ring}, $\mathrm{Gr}(\cc)$ is a fusion ring.

\end{proof}

\subsection{Examples}

\begin{example}\label{example:Vec}
It is well-known that the category $\mathrm{Vec}$ of finite dimensional $\Bbbk$-vector spaces is a fusion category, see e.g. \cite[Example 4.1.2]{Eti15}. By Proposition \ref{main result}, we know that for any non-zero algebra $A$ in $\mathrm{Vec}$, the functor $-\otimes A:\mathrm{Vec} \to \mathcal{M}_A$ is separable, where $\mathcal{M}_A$ is the category of right $A$-modules. Besides, as observed in \cite[Example 3.2 2)]{BT15}, $\Bbbk$ is a left $\otimes$-generator in $\mathrm{Vec}$. Hence, by Remark \ref{forgetful functor separable}, if $A$ is a separable algebra, then $F \circ (-\otimes A):\mathrm{Vec} \to \mathrm{Vec}$ is separable, where $F:\mathcal{M}_A \to \mathrm{Vec}$ is the forgetful functor and $-\otimes A:\mathrm{Vec} \to \mathcal{M}_A$. \end{example}

\begin{remark}
Let $G:\mathrm{Vec} \to \mathrm{Vec}$ be a right exact functor. By \cite[Theorem 2.4]{Iv12}, $G:\mathrm{Vec} \to \mathrm{Vec}$ is naturally isomorphic to $-\otimes G(\Bbbk):\mathrm{Vec} \to \mathrm{Vec}$. Since $G(\Bbbk)$ is a finite dimensional vector space, it automatically carries an algebraic structure with point-wise multiplication. By Example \ref{example:Vec}, if $G(\Bbbk)$ is a separable algebra, then $-\otimes G(\Bbbk):\mathrm{Vec} \to \mathrm{Vec}$ is separable, which implies $G:\mathrm{Vec} \to \mathrm{Vec}$ is separable. In particular, if $G$ satisfies that $G(\Bbbk) \cong \Bbbk$ i.e. $G(\Bbbk)$ is a $1$-dimensional algebra, then $G(\Bbbk)$ is a separable algebra. 
\end{remark}

\begin{example}
Let $H$ be a semisimple finite dimensional Hopf algebra over $\Bbbk$. By the reconstruction theorem, see e.g. \cite[Theorem 5.3.12]{Eti15}, we know that the representation category $\mathrm{Rep}(H)$ of $H$ is a fusion category. By Proposition \ref{main result}, we know that for any non-zero algebra $A$ in $\mathrm{Rep}(H)$, i.e. $H$-module algebra, the functor $-\otimes A:\mathrm{Rep}(H) \to \mathrm{Rep}(H)_A$ is separable.
\end{example}

\begin{example}
Let $A$ be a semisimple finite dimensional algebra over $\Bbbk$. Let
$\cc$ be the category of finite dimensional $A$-bimodules. It is known that $\cc$ is a multifusion category, and $\cc$ is a fusion category if and only if $A$ is simple, see e.g. \cite[Example 4.1.3]{Eti15}. Therefore, by Proposition \ref{main result}, we know that $A$ is simple if and only if for any non-zero algebra $A'$ in $\cc$, $-\otimes A':\cc \to \cc_{A'}$ is separable.
\end{example}

\noindent\textbf{Acknowledgements.} The author would like to thank A. Ardizzoni for meaningful suggestions and comments, and sincerely acknowledge the support provided by CSC (China Scholarship Council) through a PhD
student fellowship (No. 202406190047). This paper was written while the author was member of the ``National Group for Algebraic and Geometric Structures and their Applications'' (GNSAGA-INdAM). This work was partially supported by the project funded by the European Union - NextGenerationEU under NRRP, Mission 4 Component 2 CUP D53D23005960006 - Call PRIN 2022 No.\, 104 of February 2, 2022 of Italian Ministry of University and Research; Project 2022S97PMY \textit{Structures for Quivers, Algebras and Representations (SQUARE).}

\end{document}